\documentclass[12pt]{amsart}
\usepackage[top=1in, bottom=1in, left=1in, right=1in]{geometry}
\usepackage{amssymb,mathrsfs,amsmath}

\newtheorem{theorem}{Theorem}[section]
\newtheorem{cor}[theorem]{Corollary}
\newtheorem{lemma}[theorem]{Lemma}

\theoremstyle{definition}

\newtheorem{rk}[theorem]{Remark}

\numberwithin{equation}{section}

\newcommand {\SN} {{\mathbb N}}
\newcommand {\SR} {{\mathbb R}}

\newcommand {\SZ} {{\mathbb Z}}
\newcommand {\SP} {{\mathbb P}}

\newcommand{\bsp}{\begin{split}}
\newcommand{\esp}{\end{split}}
\newcommand{\be}{\begin{equation}}
\newcommand{\ee}{\end{equation}}
\newcommand{\bes}{\begin{equation*}}
\newcommand{\ees}{\end{equation*}}

\newcommand{\eq}[2]{ \begin{equation} \label{#1}\begin{split} #2 \end{split} \end{equation} }
\newcommand{\al}[1]{\begin{align} #1 \end{align} }
\newcommand{\als}[1]{\begin{align*} #1 \end{align*} }

\newcommand{\bv}\boldsymbol{}
\newcommand{\prob}{\textbf{Prob}}

\renewcommand{\P}{\mathcal{P}}
\newcommand{\N}{\mathcal{N}}
\newcommand{\M}{\mathcal{M}}

\newcommand{\fl}[1]{\left\lfloor#1\right\rfloor}

\newcommand{\nn}{\nonumber \\}
\numberwithin{equation}{section}

\begin{document}

\title[On the concentration of certain additive functions]{On the concentration of certain additive functions}

\author{Dimitris Koukoulopoulos}
\address{D\'epartement de math\'ematiques et de statistique\\
Universit\'e de Montr\'eal\\
CP 6128 succ. Centre-Ville\\
Montr\'eal, Qu\'ebec H3C 3J7\\
Canada}
\email{{\tt koukoulo@dms.umontreal.ca}}


\date{\today}

\maketitle

\begin{center}
\dedicatory{\textit{ \small{In memoriam Jonas Kubilius} }}
\end{center}

\begin{abstract} We study the concentration of the distribution of an additive function, when the sequence of prime values of $f$ decays fast and has good spacing properties. In particular, we prove a conjecture by Erd\H os and K\'atai on the concentration of $f(n)=\sum_{p|n}(\log p)^{-c}$ when $c>1$.
\end{abstract}


\section{Introduction}\label{intro} An arithmetic function $f:\SN\to\SR$ is called \textit{additive} if $f(mn)=f(m)+f(n)$ whenever $(m,n)=1$. According to the Kubilius probabilistic model of the integers, statistical properties of additive functions can be modeled by statistical properties of sums of independent random variables. We describe this model in the case that $f$ is a {\it strongly additive function}, that is to say $f$ satisfies the relation $f(n)=\sum_{p|n}f(p)$; the general case is slightly more involved. Let $\SP$ denote the set of prime numbers and consider a sequence of independent Bernoulli random variables $\{X_p:p\in\SP\}$ such that 
\[
\prob(X_p=1)=\frac1p\quad\text{and}\quad\prob(X_p=0)=1-\frac1p.
\]
The random variable $X_p$ can be thought as a model of the characteristic function of the event $\{n\in\SN:p|n\}$. Then a probabilistic model for $f$ is given by the random variable $\sum_pf(p)X_p$.

The above model and well-known facts from probability theory lead to the prediction that the values of $f$ follow a certain distribution, possibly after rescaling them appropriately. In fact, the Erd\H os-Wintner theorem\ \cite{ew} states that if the series 
\eq{series}{
\sum_{|f(p)|\le1}\frac{f(p)}p,
\quad \sum_{|f(p)|\le1}\frac{f^2(p)}p,
\quad \sum_{|f(p)|>1}\frac1p
} 
converge, then $f$ has a limiting distribution, in the sense that there is a distribution function $F:\SR\to[0,1]$ such that 
\[
F_x(u) := \frac{1}{\fl{x}} \{n\le x:f(n)\le u\}|  \to F(u) \quad\text{as}\quad x\to\infty
\]
for every $u\in\SR$ that is a point of continuity of $F$; the characteristic function of $F$ is given by 
\[
\hat{F}(\xi)=\prod_p\left\{\left(1-\frac1p\right) \sum_{k\ge0}\frac{e^{i\xi f(p^k)}}{p^k}\right\}.
\]
Conversely, if $f$ possesses a limiting distribution, then the three series in\ \eqref{series} converge.

One way to measure the regularity of the distribution of the set $\{f(n):n\in\SN\}$ is by its concentration. In general, given a distribution function $G:\SR\to[0,1]$, we define its concentration function to be 
\[
Q_G(\epsilon)=\sup_{u\in\SR}\{G(u+\epsilon)-G(u)\}.
\]
We seek estimates for $Q_{F_x}(\epsilon)$, or for $Q_F(\epsilon)$ if $f$ possesses a limiting distribution. There are various such results in the literature, a historic account of which is given in\ \cite{br_ten}. The most general estimate on $Q_{F_x}(\epsilon)$ is due to Ruzsa\ \cite{ruzsa}. Improving upon bounds due to Erd\H os\ \cite{erd1} and Hal\'asz\ \cite{halasz}, he showed that 
\eq{intro_e1}{
Q_{F_x}(1)
	\ll  \max_{\lambda\in\SR} \frac{1}{ \sqrt{\lambda^2 + 
		\sum_{p\le x} \min\{1,(f(p)-\lambda\log p)^2\}/p } }.
}
This result is best possible, as can be seen by taking $f(n)=c\log n$ or $f(n)=\omega(n)=\sum_{p|n}1$. However, both of these functions satisfy $f(p)\gg1$. So, a natural question is whether it is possible to improve upon\ \eqref{intro_e1} in the case that $f(p)$ decays to zero. Erd\H os and K\'atai\ \cite{ek}, building on earlier work of Tjan\ \cite{tjan} and Erd\H os\ \cite{erd2}, showed the following result:


\begin{theorem}[Erd\H os, K\'atai\ \cite{ek}]\label{ekthm}  
Let $f:\SN\to\SR$ be an additive function such that 
\[
\sum_{p>t^A}\frac{|f(p)|}p\ll\frac1t   \quad(t\ge1), 	
\quad
|f(p_1)-f(p_2)|\gg\frac1{p_2^B}   \quad(p_1,p_2\in\SP,\,p_1<p_2),
\] 
for some constants $A$ and $B$. Then
\[
Q_F(\epsilon) \asymp_{A,B} \frac1{\log(1/\epsilon)} \quad(0<\epsilon\le1/2);
\] 
except for $A$ and $B$, the implied constant depends on the implied constants in the assumptions of the theorem too.
\end{theorem}


\noindent On the other hand, when $f(p)\ll1/p^\delta$, $p\in\SP$, for some $\delta>0$, then \eqref{intro_e1} applied to $f/\epsilon$ yields an upper bound for $Q_F(\epsilon)$ that is never better than $1/\sqrt{\log\log(1/\epsilon)}$, as can be seen by taking $\lambda=0$.

Also, Erd\H os and K\'atai studied $Q_F(\epsilon)$ in the case that $f(p)=(\log p)^{-c}$, $p\in\SP$, for some $c\ge1$. They showed that 
\eq{eke almost}{
\begin{cases}
\epsilon^{1/c} \ll_c Q_F(\epsilon) \ll_c\epsilon^{1/c} \log\log^2(1/\epsilon)&\text{if}~c>1,\cr 
\epsilon \ll Q_F(\epsilon) \ll \epsilon\log(1/\epsilon) \log\log ^2(1/\epsilon) & \text{if}~c=1,
\end{cases}
}
for $0<\epsilon\le1/3$. Furthermore, they conjectured that, for fixed every $c>1$, we have that
\[
Q_F(\epsilon)\asymp_c\epsilon^{1/c}   \quad(0<\epsilon\le1).
\]

The conjecture of Erd\H os and K\'atai was proven for $c$ large enough by La Bret\`eche and Tenenbaum in\ \cite{br_ten}:


\begin{theorem}[La Bret\`eche, Tenenbaum\ \cite{br_ten}]\label{bt} Let $c\ge1$ and $f:\SN\to\SR$ be an additive function such that $|f(p)|\asymp(\log p)^{-c}$ for every $p\in\SP$ and 
\[
|f(p_1)-f(p_2)| \gg  \frac{p_2-p_1}{p_2(\log p_2)^{c+1}}  \quad(p_1,p_2\in\SP,\,p_1<p_2).
\]
If $c$ is large enough, then we have that 
\[
Q_F(\epsilon) \asymp \epsilon^{1/c}\quad(0<\epsilon\le1);
\]
the implied constant depends at most on the implied constants in the assumptions of the theorem.
\end{theorem}


La Bret\`eche and Tenenbaum derived their theorem from a general upper bound on $Q_F(\epsilon)$ that they showed when the sequence of prime values of $f$ satisfies certain regularity assumptions. Their method uses a result from the theory of functions of \textit{Bounded Mean Oscillation}, first introduced by Diamond and Rhoads\ \cite{dr} in this context to study the concentration of $f(n)=\log(\phi(n)/n)$.

In this paper we give a proof of the full Erd\H os-K\'atai conjecture using a more elementary method, similar to the ones in\ \cite{erd2,ek}:


\begin{theorem}\label{ekc} Let $c\ge1$ and $f:\SN\to\SR$ be an additive function with $f(p)=(\log p)^{-c}$ for all $p\in\SP$. For $0<\epsilon\le1/2$ we have that 
\[
\epsilon^{1/c}\ll Q_F(\epsilon)\ll\min\left\{\frac{c}{c-1},\log\frac1\epsilon\right\}\epsilon^{1/c}.
\]
\end{theorem}


\begin{rk}\label{rk1} 
When $0<c<1$, the behavior of $Q_F$ for $f$ as in Theorem\ \ref{ekc} is different. As G\'erald Tenenbaum has pointed out to us in a private communication, in this case we have that 
\eq{c<1}{
Q_F(\epsilon) \asymp_c \epsilon   \quad(0<\epsilon\le1).
} 
Corollary\ \ref{ekc} and relation\ \eqref{c<1} give the concentration of an additive function $f$ with $f(p)=(\log p)^{-c}$, $p\in\SP$, for all positive values of $c$ except for $c=1$, which is the only case remaining open.
\end{rk}


We will prove Theorem\ \ref{ekc} in Section\ \ref{ekc proof}. The method of its proof is quite flexible; in particular, it leads to a strengthening of Theorems\ \ref{ekthm} and\ \ref{bt}. We phrase our more general result in terms of the distribution function 
\[
\mathcal{F}_y(u)
	= \prod_{p\le y}\left(1-\frac1p\right) \sum_{\substack{p|n\Rightarrow p\le y\\f(n)\le u}} \frac1n \quad  (u\in\SR),
\]
defined for every $y\ge1$. From a technical point of view, this function is more natural to work with than $F_x$. Indeed, a calculation of the characteristic function of $\mathcal{F}_y$ immediately implies that $\mathcal{F}_y$ converges to $F$ weakly, provided that the latter is well defined. It is relatively easy to pass from estimates for $Q_{\mathcal{F}_y}(\epsilon)$ to estimates for $Q_{F_x}(\epsilon)$. 

With this notation, we have the following result (observe that by letting $y\to\infty$ in it, we deduce as special cases\footnote{To deduce Theorem\ \ref{ekthm}, take $\P=\{p\in\SP:|f(p)|\le p^{-1/(2A)}\}$.} Theorems\ \ref{ekthm} and\ \ref{bt}):


\begin{theorem}\label{main thm} Consider an additive function $f:\SN\to\SR$ for which there is a set of primes $\P$ and a constant $c\in[1,2]$ such that 
\[
|f(p)| \ll \frac1{(\log p)^c}\quad(p\in\P),
\quad\text{and}\quad
\sum_{p\in\SP\setminus\P}\frac1p\ll1.
\]
For $t\ge2$ set 
\[
g(t) = \frac{\sup\{|f(p)|(\log p)^c:p\ge t,\,p\in\P\}}{(\log t)^c}
\]
and assume that there is some $A\ge1$ such that 
\[
|f(p_2)-f(p_1)| \gg \min\left\{\frac{g(p_2)(p_2-p_1)}{p_2\log p_2}, g\left(p_2^A\right)\right\}
\quad(p_1,p_2\in\P,\,p_1<p_2).
\]
Then for $0<\epsilon\le1/2$ and $y\ge K(\epsilon)$, where $K(\epsilon)=\min \{ p\in\SP : g(p) \le \epsilon\}$, we have that 
\[
\frac1{\log K(\epsilon)} 
	\ll Q_{\mathcal{F}_y}(\epsilon)
	\ll_A \frac{\min\left\{1/(c-1), \log(1/\epsilon)  \right\}}{\log K(\epsilon)};
\]
except for $A$, the implied constants depends on the implied constants in the assumptions of the theorem too.
\end{theorem}


As an immediate corollary, we deduce the following simpler to state result.


\begin{cor}\label{main cor} Let $f:\SN\to\SR$ be an additive function for which there is a constant $c\in[1,2]$ such that the sequence $\{|f(p)|(\log p)^c:p\ \text{prime}\}$ is decreasing. Then for $0<\epsilon\le1/2$ and $y\ge K(\epsilon)$, where $K(\epsilon)=\min \{ p\in\SP:   |f(p)| \le \epsilon\}$, we have that 
\[
\frac1{\log K(\epsilon)} 
	\ll Q_{\mathcal{F}_y}(\epsilon)
	\ll \frac{\min\left\{1/(c-1), \log(1/\epsilon)  \right\}}{\log K(\epsilon)} .
\]
\end{cor}


\begin{proof} If $g$ is as in the statement of Theorem \ref{main thm} with $\P = \SP$, then we see immediately that $g(p)=|f(p)|$, for all $p\in\SP$. Moreover, if $p_1<p_2$ are two primes, then we have that
\als{
\frac{|f(p_2)-f(p_1)|}{|f(p_2)|}
	\ge \frac{|f(p_1)|}{|f(p_2)|}  - 1 
	\ge \frac{(\log p_2)^c}{(\log p_1)^c} - 1 
	&= \frac{ (\log p_2)^c - (\log p_1)^c}{(\log p_1)^c} \\
	&\ge \frac{c(p_2-p_1)(\log p_1)^{c-1}}{p_2(\log p_1)^c} \\
	&\ge \frac{p_2-p_1}{p_2\log p_2} ,
}
by our assumption that $\{|f(p)|(\log p)^c:p\ \text{prime}\}$ is decreasing and the Mean Value Theorem. So Theorem \ref{main thm} can be applied and the claimed result follows.
\end{proof}


The lower bound in Theorem \ref{main thm}, which will be proven in Section \ref{lower}, is a straightforward application of Theorem 1.2 in\ \cite{br_ten}. On the other hand, for the proof of the upper bound in Theorem \ref{main thm}, which will be given in Section \ref{upper}, we use a combination of ideas from\ \cite{erd2,ek}. Even though Theorem\ \ref{ekc} is an immediate corollary of Theorem\ \ref{main thm}, applied with $\min\{c,2\}$ in place of $c$, we have chosen to give the proof of both of them in full detail, so that to motivate certain choices in the proof of Theorem\ \ref{main thm}, which is rather technical. 


\subsection*{A heuristic argument} There is a simple heuristic argument which motivates Theorem\ \ref{main thm}. We demonstrate it in the simpler setting of Corollary \ref{main cor}, that is to say when the sequence $\{|f(p)|(\log p)^c:p\in\SP\}$ is decreasing. For every integer $n$, we have that 
\[
\sum_{p|n,\,p\ge K(\epsilon)}|f(p)| \le \epsilon \sum_{p|n,\,p\ge K(\epsilon)}\frac{(\log K(\epsilon))^c}{(\log p)^c}.
\]
Since for a typical integer $n$ the sequence $\{\log\log p:p|n\}$ is distributed like an arithmetic progression of step 1 (see, for example, \cite[Chapter 1]{ht}), we find that\footnote{The symbol `$\lesssim$' here is used in a non-rigorous fashion to denote `roughly less than'. Similarly, the symbol `$\approx$' means `roughly equal to'.} 
\[
\sum_{p|n,\,p\ge K(\epsilon)}|f(p)
	|\lesssim \epsilon \sum_{j\ge\log\log K(\epsilon)}\frac{(\log K(\epsilon))^c}{e^{cj}}
	\ll\epsilon.
\]
So only the prime divisors of $n$ lying in $[1,K(\epsilon))$ are important for the size of $Q_F(\epsilon)$. Note that for a prime number $p<K(\epsilon)$, we have that $|f(p)|>\epsilon$. Therefore if $a$ and $b$ are composed of primes within $[1,K(\epsilon))$, then it is reasonable to expect that $|f(a)-f(b)|$ is big compared to $\epsilon$, unless $a$ and $b$ have a large common factor. This leads to the prediction that $Q_F(\epsilon)\approx1/\log K(\epsilon)$, which is confirmed by Theorem\ \ref{main thm} when $c>1$. However, when $c<1$ this heuristic fails, as\ \eqref{c<1} shows, and the underlying reason is combinatorial: the pigeonhole principle implies the lower bound $Q_F(\epsilon)\gg_F\epsilon$ for the concentration function of any distribution function $F$ (see also\ \cite[Remark 1, p. 297]{ek}).


\subsection*{Notation} For an integer $n$ we denote with $P^+(n)$ and $P^-(n)$ its largest and smallest prime factors, respectively, with the notational convention that $P^+(1)=1$ and $P^-(1)=\infty$. The symbols $p$ and $p'$ always denote prime numbers. The set of all primes numbers is denoted by $\mathbb{P}$. Finally, given $\P\subset \SP$ and real numbers $1\le z\le w$, we write $\P(z,w)$ for the set of integers all of whose prime factors belong to $\P\cap (z,w]$.


\section{The conjecture of Erd\H os and K\'atai}\label{ekc proof}

\begin{proof}[Proof of Theorem\ \ref{ekc}] The lower bound follows by relation \eqref{eke almost}, with the implied constant depending on $c$. To remove this dependence, see Theorem\ \ref{main-lower} below. 

It remains to show the corresponding upper bound. Before delving into the details of the proof, we give a brief outline of the main idea. For $\delta>0$, we set $P_{\delta} = \exp\{\delta^{-1/c}\}$, so that $f(p)=\delta$ if and only if $p=P_\delta$. As the heuristic argument presented towards the end of Section \ref{intro} indicates, it suffices to bound $Q_{\mathcal{F}_q}(\epsilon)$, where $q:=P_{2\epsilon}$. We split the elements of the set $\M: =\{n\in\SP(1,q): u<f(n)\le u+\epsilon\}$ into subsets $\M_\delta:=\{n\in\M: P_{2\delta}< P^-(n) \le P_\delta \}$, where $\delta \in \{ 2^j \epsilon:1\le j\le j_0\}$ with $j_0=\fl{(\log(1/\epsilon) - c\log\log 2)/\log 2}\}$. Then we find that
\[
\sum_{n\in \M_\delta} \frac{1}{n} 
	\approx \sum_{m\in\SP(P_{\delta},q)}  \frac{1}{m} 
		\sum_{\substack{ P_{2\delta}<p\le P_\delta \\ u-f(m) <f(p) \le u-f(m)+\epsilon}} \frac{1}{p}.
\]
Fix $m$ for the moment and set $v_m=u-f(m)$. Then the variable $p$ lies in the interval $I_m=[P_{\min\{v_m+\epsilon,2\delta\}}, P_{\max\{v_m,\delta\}}]$, which is non-empty only when $v_m+\epsilon\ge \delta\ge v_m/2$. Since $\delta\ge2\epsilon$ by assumption, we find that $v_m\asymp \delta\gg\epsilon$. So the interval $I_m$ has double-logarithmic length\footnote{Given an interval $I=[\alpha,\beta]$, its double-logarithmic length is $\log\log \beta - \log\log \alpha$.} $\ll\epsilon/v_m \asymp \epsilon/\delta$. Hence the Prime Number Theorem \cite[Theorem 1, p. 167]{ten_book} implies that $\sum_{p\in I_m}1/p\lesssim \epsilon/\delta$, provided that $I$ is not too short. Assuming that this is indeed the case, we deduce that
\[
\sum_{n\in \M_\delta} \frac{1}{n} 
	\lesssim \frac{\epsilon}{\delta} \sum_{m\in\SP(P_{\delta},q)}  \frac{1}{m} 
	\ll \frac{\epsilon\log q}{\delta \log P_{\delta}} \asymp \frac{\epsilon^{1-1/c}}{\delta^{1-1/c}} .
\]
Summing the above inequality over $\delta\in  \{ 2^j \epsilon: 1\le j \le j_0\}$ implies that
\als{
Q_F(\epsilon) 
	\approx Q_{\mathcal{F}_q}(\epsilon) 
	\lesssim  \frac{1}{\log q} \sum_{j=1}^{j_0} \sum_{n\in \M_{2^j\epsilon}} \frac{1}{n}
		&\lesssim \epsilon^{1/c} \sum_{j=1}^{j_0} \frac{1}{2^{j(1-1/c)}} 	 \\
		& \ll \min\left\{\frac{c}{c-1}, \log\frac{1}{\epsilon} \right\} \epsilon^{1/c},
}
which shows (heuristically at least) the desired result. 

The main technical difficulty we have to surpass in order to make the above argument work is that the estimate $\sum_{p\in I_m} 1/p\ll \epsilon/\delta$, which we used above, might not be accurate for large $\delta$ (i.e. when $P_\delta$ is small). So below we shall employ a variation of the argument of this paragraph where, instead of looking where $P^-(n)$ lies, we will look where $\min\{p|n:p\ge q'\}$ lies, with $q'$ being some small parameter chosen appropriately.

\medskip

Without loss of generality, we may assume that $\epsilon\le1/100^c$; otherwise\ \eqref{ekc goal} follows immediately by the trivial bound $Q_F(\epsilon)\le 1$. Define $\eta$ by $P_\eta/\eta^2 = 1/\epsilon^2$. Note that $4\epsilon\le \eta\le 1/2$, since $P_{4\epsilon}/(4\epsilon)^2\ge 1/\epsilon^2\ge P_{1/2}/(1/2)^2$ by our assumption that $\epsilon\le 1/100^c$. Before we proceed further, we will prove that, for $v\in\SR$, $\delta\in[2\epsilon,1]$ and $2\le z\le P_\delta$, we have that
\eq{ekc main eq}
{
\sum_{ \substack{ z<p\le P_\delta \\ v<f(p)\le v+\epsilon} } \frac1p
	\ll 	\frac{\epsilon}{\delta} + \frac{1}{\sqrt{z}} 
	\ll 	\begin{cases}
			\epsilon/\delta 				&\text{if}\ z\ge P_{2\delta} \ge P_\eta,\cr 
			\epsilon/\delta+ 1/(\log z)^{2}	&\text{otherwise}.
		\end{cases}
}
First, note that it suffices to show the first inequality. Indeed, if $z\ge P_{2\delta}\ge P_\eta$, then $P_{2\delta}/(2\delta)^2\ge P_\eta/\eta^2=1/\epsilon^2$ and thus $\sqrt{z}\ge \sqrt{P_{2\delta}} \ge 2\delta/\epsilon$, which proves the second inequality of \eqref{ekc main eq}. Turning back to the first inequality of \eqref{ekc main eq}, observe that the primes $p$ on the left hand side of \eqref{ekc main eq} lie in the interval $[\max\{z,P_{v+\epsilon}\}, P_{\max\{v,\delta\}}] =: [\alpha,\beta]$. For this interval to be non-empty we need that $v+\epsilon\ge \delta$. Since $\delta\ge2\epsilon$, we deduce that $v\ge \epsilon$ and thus $2v\ge v+\epsilon \ge\delta$. So we have that
\[
\log\left( \frac{\log \beta}{\log \alpha} \right)
	\le \log\left(\frac{\log P_v}{\log P_{v+\epsilon}} \right) 	
	= \frac{1}{ c} \log \left( \frac{v+\epsilon}{v} \right) 
	\le \frac{\epsilon}{v} 
	\le \frac{2\epsilon}{\delta} .
\]
Thus, if $\beta\ge\alpha+\sqrt{\alpha}$, then covering the interval $[\alpha,\beta]$ by subintervals of the form $[y,y+\sqrt{y})$ and applying the Brun-Titchmarsch inequality\ \cite[Theorem 9, p. 73]{ten_book} to each one of them yields that
\[
\sum_{ \substack{ z<p\le P_\delta \\ v<f(p)\le v+\epsilon} } \frac1p \ll \frac{\epsilon}{\delta} ,
\]
which proves \eqref{ekc main eq}. Finally, if $\beta<\alpha+\sqrt{\alpha}$, then we have that
\[
\sum_{ \substack{ z<p\le P_\delta \\ v<f(p)\le v+\epsilon} } \frac1p
	\le \sum_{\alpha\le p\le \beta}\frac{1}{p}
	\le \frac{\beta-\alpha+2}{\alpha}\ll \frac{1}{\sqrt{\alpha}} \le \frac{1}{\sqrt{z}}
\]
and \eqref{ekc main eq} follows in this last case too.

We are now ready to show the upper bound implicit in Theorem\ \ref{ekc}. Fix for the moment $y\ge q = P_{2\epsilon}$ and $u\in\SR$. Given $n\in \SP(1,y)$ with $u<f(n)\le u+\epsilon$, we write $n=ab$, where $a$ is square-free, $b$ is square-full and $(a,b)=1$. We further decompose $a=a_1a_2$, where $P^+(a_1)\le q<P^-(a_2)$. So
\als{
\sum_{\substack{P^+(n)\le y \\ u<f(n)\le u+\epsilon }} \frac1n
	&= \sum_{ \substack{ P^+(b)\le y \\ b\ \text{square-full}} } \frac{1}{b} 
	 	\sum_{\substack{ a_2\in \SP(q,y) \\ (a_2,b)=1} } \frac{\mu^2(a_2)}{a_2}	
		\sum_{\substack{P^+(a_1)\le q,\ (a_1,b)=1 \\ u - f(a_2b) < f(a_1) \le u - f(a_2b) + \epsilon }} 
		\frac{\mu^2(a_1)}{a_1} \\
	&\le  \sum_{ \substack{ P^+(b)\le y \\ b\ \text{square-full}} } \frac{1}{b} 
	 	\sum_{a_2\in \SP(q,y) } \frac{1}{a_2} 
		\sup_{v\in\SR}\left\{\sum_{\substack{P^+(a_1)\le q \\ v < f(a_1)\le v+\epsilon}}
			\frac{\mu^2(a_1)}{a_1} \right\} \nn
	&\ll \frac{\log y}{\log q} \cdot   \sup_{v\in\SR}\left\{\sum_{\substack{ P^+(a_1) \le q \\v<f(a_1)\le v+\epsilon}}
		\frac{\mu^2(a_1)}{a_1}\right\} . 
}
Since $\log q \asymp \epsilon^{-1/c}$, then Theorem \ref{ekc} will follow from the estimate
\eq{ekc goal}{
\sum_{\substack{P^+(n)\le q \\v<f(n)\le v+\epsilon}}\frac{\mu^2(n)}{n} \ll \min\left\{\frac{c}{c-1},\log\frac1\epsilon\right\}
	\quad (v\in \SR)
}
by letting $y\to\infty$. Set $J=-1+ \fl{\log(\eta/\epsilon) /\log 2 }\in\SN$ and, for $j\ge0$, define $q_j = P_{2^{j+1}\epsilon}$, so that $q=q_0>\cdots >q_J\ge P_\eta$. Fix $v\in\SR$ and let 
\[
\N = \{ n\in\SN: \mu^2(n)=1,\ P^+(n)\le q_0,\ v<f(n)\le v+\epsilon \}.
\]
As in the heuristic argument of the first paragraph, we partition $\N$ into certain subsets and estimate the contribution of each one of them to $\sum_{n\in\N}1/n$ separately. The difference is that instead of looking at the location of $P^-(n)$, we write $n=an'$ with $P^+(a)\le q_J<P^-(n')$ and look at the location of $p=P^-(n')$. An additional fact that we shall take advantage of is that if $n'=pb$ and $\log p\asymp \log P_\delta$, then, for fixed $b$, the number $f(a)$ lies in an interval of length $\ll \epsilon+f(p)\ll \epsilon+\delta\ll \delta$, which allows us to gain an additional crucial savings in our estimate for $\sum_{n\in \N}1/n$. So we write $\N=\cup_{j=0}^J\N_j $, where $\N_0=\{n\in\N : P^+(n)\le q_J\}$ and 
\[
\N_j = \{n\in \N: n=apb,\ P^+(a)\le q_J<p<P^-(b),\ q_j<p\le q_{j-1}\}
\]
for $j\in\{1,\dots, J\}$.

First, we bound $\sum_{n\in\N_0}1/n$. If $n>1$, then we write $n=mP^+(n)=mp'$. Thus
\al{
\sum_{n\in\N_0}\frac1n
	&\le 1 + \sum_{P^+(m)\le q_J}\frac1m
		\sum_{\substack{ P^+(m)< p' \le q_J \\ v-f(m)<f(p') \le v-f(m)+\epsilon}} \frac1{p'} \label{ekc e10a} \\
	&\ll  1 + \sum_{P^+(m)\le q_J} \frac1m \left( \frac{\epsilon}{2^J\epsilon} + \frac1{\log^2(1+P^+(m))}\right)
		\ll1+ \frac{\log q_J}{2^J}   \label{ekc e10b},
}
by relation \eqref{ekc main eq}. 

Next, we bound $\sum_{n\in\N_j}1/n$ for $j\in\{1,\dots,J\}$. We have that 
\eq{ekc e3}{
\sum_{n\in\N_j}\frac1n 
	&\le\sum_{b\in\SP(q_j,q_0)} \frac1b 
		\sum_{P^+(a)\le q_J} \frac{\mu^2(a)}{a}
		\sum_{ \substack{ q_j < p\le q_{j-1} \\ v-f(a) - f(b)<f(p) \le v - f(a) - f(b) + \epsilon} }
		\frac1p\\
	&\le	\sum_{b\in\SP(q_j,q_0)} \frac1b
			\sup_{w\in\SR} \left\{ 		\sum_{P^+(a)\le q_J} \frac{\mu^2(a)}{a}
			\sum_{\substack{ q_j < p\le q_{j-1} \\ w-f(a) < f(p) \le w - f(a) + \epsilon}}	
			\frac1p\right\} \\
	&\ll  \frac{\log q_0}{\log q_j} \cdot 	\sup_{w\in\SR} \left\{ 		\sum_{P^+(a)\le q_J} \frac{\mu^2(a)}{a}
			\sum_{\substack{ q_j < p\le q_{j-1} \\ w-f(a) < f(p) \le w - f(a) + \epsilon}}	
			\frac1p\right\} \\	
}
Fix some $w\in\SR$ and consider $a$ with $P^+(a)\le q_J$ and $p\in(q_j,q_{j-1}]$ with $w<f(a)+f(p)\le w+\epsilon$, as above. Since $|f(p)|\le2^{j+1}\epsilon$ for $p>q_j$, we must have that $|f(a) -w|<2^{j+2}\epsilon$. So 
\al{
\sum_{P^+(a)\le q_J} \frac{\mu^2(a)}{a}
		\sum_{\substack{ q_j < p\le q_{j-1} \\ w-f(a)<f(p)\le w-f(a)+\epsilon }} \frac1p 
	&= \sum_{ \substack{ P^+(a)\le q_J\\ |f(a) - w| < 2^{j+2}\epsilon } }\frac{\mu^2(a)}{a}
		\sum_{\substack{ q_j < p\le q_{j-1} \\ w - f(a) < f(p) \le w - f(a) + \epsilon } }  \frac1p \nn
	&\ll \frac{\epsilon} {2^j\epsilon} 
		\sum_{ \substack{ P^+(a)\le q_J\\  |f(a) - w| < 2^{j+2}\epsilon} }\frac{\mu^2(a)}{a},  \label{ekc e4}
}
by the first part of\ \eqref{ekc main eq} applied with $w-f(a)$, $2^{j}\epsilon$ and $q_j$ in place of $v$, $\delta$ and $z$, respectively, since $q_j\ge q_J \ge P_\eta$. Finally, if $a>1$, then we write $a=mP^+(a)=mp'$. So we find that
\[
\sum_{ \substack{ P^+(a)\le q_J\\  |f(a) - w| < 2^{j+2}\epsilon} }\frac{\mu^2(a)}{a}
	\le1 + \sum_{P^+(m)\le q_J} \frac1m
	\sum_{\substack{P^+(m) < p' \le q_J \\ |f(p') - (w - f(m)) | < 2^{j+2}\epsilon }} \frac1{p'} .
\]
For every fixed $m\in\SN$, we have that 
\[
\sum_{ \substack{ P^+(m)< p' \le q_J \\  |f(p') - (w - f(m)) | < 2^{j+2}\epsilon}} \frac1{p'}
\ll \frac{2^j\epsilon}{2^J\epsilon} + \frac1{\log^2(1+P^+(m))},
\]
by the second part of\ \eqref{ekc main eq} with $2^j\epsilon$, $2^{J+1}\epsilon$ and $P^+(m)$ in place of $\epsilon,\delta$ and $z$, respectively\footnote{Note that the parameter $\eta$ is not involved in the second part, so the same proof allows us to replace $\epsilon$ with $2^j\epsilon$.}, and with $v\in\{w-f(m) + h\cdot 2^j\epsilon: h\in[-4,4)\cap\SZ\}$. So we find that 
\[
\sum_{\substack{ P^+(a) \le q_J \\ |f(a) -w| <2^{j+1}\epsilon }}\frac{\mu^2(a)}{a}
	\ll 1 + \sum_{P^+(m)\le q_J} \frac1m \left( \frac{2^j\epsilon}{2^J\epsilon} + \frac1{\log^2(1+P^+(m))}\right)
	\ll 1 + \frac{\log q_J}{2^{J-j}}.
\]
Combining the above inequality with\ \eqref{ekc e3} and\ \eqref{ekc e4}, we deduce that
\[
\sum_{n\in\N_j} \frac1n 
	\ll \frac{1}{2^j} \left(1+\frac{\log q_J} {2^{J-j}} \right)  \frac{\log q_0}{\log q_j}
	\ll \frac{1}{2^{j(1-1/c)}} \left(1+\frac{\log q_J} {2^{J-j}} \right) .
\]
Together with relation \eqref{ekc e10b}, this implies that
\als{
\sum_{n \in \N }\frac1{n} 
	\ll \sum_{j=0}^J  \frac{1}{2^{j(1-1/c)}} \left(1+\frac{\log q_J} {2^{J-j}} \right) 
	&= \sum_{j=0}^J  \frac{1}{2^{j(1-1/c)}} + \sum_{j=0}^J \frac{2^{j/c} \log q_J}{2^J}  \\
	&\ll \min\left\{\frac{c}{c-1},J\right\} + \frac{\min\{c,J\} \log q_J} {2^{J(1-1/c)}}  . 
}
Furthermore, we have that $2^J\asymp \eta/\epsilon$ and, as a result,
\[
\log q_J=(2^{J+1}\epsilon)^{-1/c}\asymp \eta^{-1/c}= \log P_\eta = 2\log(\eta/\epsilon)\ll J \ll \log(1/\epsilon).
\]
Thus
\als{
\sum_{ n\in\N }\frac1{n} 
		\ll \min\left\{\frac{c}{c-1},J\right\} + \frac{J\min\{c,J\}}{2^{J(1-1/c)}}
	&\ll \min\left\{\frac{c}{c-1},J\right\} \\
	&\ll  \min\left\{\frac{c}{c-1},\log\frac1\epsilon\right\} ,
}
by the inequality $2^{J(1-1/c)}\gg J^2$ if $c\ge2$ and the inequality $2^{J(1-1/c)}\gg\max\{1,J(1-1/c)\}$ if $1\le c\le 2$. Therefore relation\ \eqref{ekc goal} follows, thus completing the proof of the theorem.
\end{proof}


\section{The lower bound in Theorem \ref{main thm}}\label{lower}

In this section we derive the lower bound in Theorem \ref{main thm} from the following general result, which is a corollary of Theorem 1.2 in\ \cite{br_ten}.


\begin{theorem}\label{main-lower} Let $f:\SN\to\SR$ be an additive function and $0<\epsilon<1$. If there is a set of primes $\P$ and some $M\ge2$ such that 
\[
\sum_{p\in\SP\setminus\P} \frac1p
	\ll1\quad\text{and}\quad\sum_{p\in\P,\,p>M}\frac{|f(p)|}p\ll\epsilon,
\]
then, for $y\ge M$, we have that 
\[
Q_{\mathcal{F}_y}(\epsilon)\gg\frac1{\log M};
\]
the implied constant depends at most on the implied constants implicit in the assumptions of the theorem.
\end{theorem}


\begin{proof} Let $M_0$ be a large constant to be chosen later. If $M\le y\le M_0$, then the theorem follows by the trivial bound $Q_{\mathcal{F}_y}(\epsilon)\gg1/\log y$, which holds since 1 is always in $\{n\in\SN: P^+(n)\le y,\,|f(n)|<\epsilon/2\}$. Assume now that $y\ge M_0$ and set $M'=\max\{M,M_0\}$, so that $y\ge M'$. Let
\[
C = \frac{1}{\epsilon} \sum_{p\in\P,\,p>M'}\frac{|f(p)|}p\ll1.
\] 
Define $g:\SN\to\SR$ by 
\[
g(n) =   \begin{cases}f(n)&\text{if}\ n\in\P(1,y),\cr
		 0 & \text{otherwise},
		\end{cases}
\] 
and call $G$ its distribution function. Then Theorem 1.2 in\ \cite{br_ten} yields that\footnote{In \cite[Theorem 1.2]{br_ten}, the authors let $\epsilon\to0$. However, an easy modification of their proof allows us to let instead $M'\to\infty$.}
\[
Q_{G}(3C\epsilon)
	\ge \left( 1-2 \cdot \frac{C\epsilon}{3C\epsilon} + o_{M'\to\infty}(1) \right) \prod_{p\le M'}\left(1-\frac1p\right)
	\gg  \frac1{\log M'} .
\]
So, by the pigeonhole principle, we deduce that 
\eq{pp}{
Q_G(\epsilon)	\ge  \frac{Q_G(3C\epsilon)}{3C+1} \gg_C\frac1{\log M'}  \asymp \frac{1}{\log M} ,
}
provided that $M_0$ is large enough. Finally, we have that 
\als{
Q_G(\epsilon) 
	&= \sup_{u\in\SR} \left\{ \prod_{p\in\P\cap[1,y]} \left(1-\frac1p\right)  
		\sum_{\substack{ n\in\P(1,y)   \\u<f(n)\le u+\epsilon}} \frac1n\right\} \\
	&\le Q_{\mathcal{F}_y}(\epsilon)\prod_{p\in\SP\setminus\P}\left(1-\frac1p\right)^{-1} 
		\ll Q_{\mathcal{F}_y}(\epsilon),
}
which together with\ \eqref{pp} completes the proof of the theorem.
\end{proof}


\begin{proof}[Proof of the lower bound in Theorem \ref{main thm}] The definition of $g$ implies that the function $t\to g(t)(\log t)^c$ is decreasing. Thus, for $p\in\P$ with $p\ge K(\epsilon)$, we have that 
\[
|f(p)|\le g(p)
	\le \frac{g(K(\epsilon))(\log(K(\epsilon)))^c}{(\log p)^c}
	\le\frac{\epsilon(\log(K(\epsilon)))^c}{(\log p)^c}.
\]
Consequently,
\[
\sum_{p>K(\epsilon)} \frac{|f(p)|}{p} \ll \epsilon,
\] 
which implies that the hypotheses of Theorem\ \ref{main-lower} are satisfied with $M=K(\epsilon)$ and $\P$, and the desired lower bound follows.
\end{proof}


\section{The upper bound in Theorem \ref{main thm}}\label{upper}

We conclude the paper by showing the upper bound in Theorem \ref{main thm}. We start with the following technical lemma whose hypotheses mimic all the crucial facts about the additive function $f(n)=\sum_{p|n}(\log p)^{-c}$ that we used in the proof of Theorem \ref{ekc}.


\begin{lemma}\label{main-upper} Let $f:\SN\to\SR$ be an additive function for which there is a set of primes $\P$ and a decreasing function $P_f:(0,1]\to[2,+\infty)$ such that
\eq{c0}{
\sum_{p\in\SP\setminus\P}\frac1p\ll1,
}
\eq{c1}{
|f(p)|\le\epsilon\quad(0<\epsilon\le1,\,p\in\P,\,p>P_f(\epsilon)).
}
Furthermore, assume that there is some $\lambda\in(0,1]$ and some $\rho\ge1$ such that
\eq{c2}{
\sum_{ \substack{ p\in\P\cap(z,w] \\ u < f(p) \le u + \epsilon }} \frac1p
	\ll \begin{cases}
		\epsilon/\delta				&\text{if}\ z\ge P_f(2\delta) \ge (2\delta/\epsilon)^\rho ,\cr 
		\epsilon/\delta+ 1/(\log z)^2	&\text{otherwise},
	\end{cases}
}
for all $u\in\SR$, $0<\epsilon \le \delta \le1$ and $2\le z\le w\le \min\{ P_f(\delta),P_f(\epsilon)^\lambda \}$. Let $0<\epsilon\le\delta\le1$ such that $P_f(\delta)\le P_f(\epsilon)^\lambda$, set $q_j=P_f(2^j\delta)$ for $j\ge0$, and consider $J\in\{0\}\cup\{j\in\SN:2^j\le1/\delta\ \text{and}\ q_j\ge(2^j\delta/\epsilon)^\rho\}.$ For $y\ge q_0$, we have that \eq{e300}{ 
Q_{\mathcal{F}_y}(\epsilon) \ll \frac1{\log q_0} 
		+ \frac{\epsilon}{\lambda\delta} \sum_{j=1}^J\frac1{2^j\log q_j}
		+ \frac{\epsilon}{\delta} \sum_{j=0}^J\frac{\log q_J}{2^J\log q_j};
} 
the implied constant depends at most on the implied constants in\ \eqref{c0} and\ \eqref{c2}.
\end{lemma}


\begin{rk} The parameters $\lambda$ and $\rho$, and the set $\P$ are introduced to make Lemma\ \ref{main-upper} more applicable. One can think of $P_f$ defined by $P_f(\epsilon)=\max\{p\in\SP:|f(p)|>\epsilon\}$. Condition\ \eqref{c2} can be motivated as follows. Assume that 
\eq{rke1}{
\sum_{p>P_f(\alpha)} \frac{|f(p)|}p  \ll \alpha \quad(0<\alpha\le1).
}
We have that $|f(p)|\approx\delta$ for $p\in(P_f(2\delta),P_f(\delta)]$. So if the sequence $\{f(p):p\in\SP\}$ is `well-spaced', then we expect that 
\[
\sum_{\substack{ P_f(2\delta) < p \le P_f(\delta) \\ w<f(p)\le w+\epsilon }} \frac1p 
	\lesssim \frac{\epsilon}{\delta} \sum_{P_f(2\delta) < p \le P_f(\delta) } \frac1p
	\approx \frac{\epsilon}{\delta^2} \sum_{P_f(2\delta) < p \le P_f(\delta)} \frac{|f(p)|}p
	\ll \frac{\epsilon}{\delta},
\] 
by\ \eqref{rke1}.
\end{rk}


\begin{proof}[Proof of Lemma\ \ref{main-upper}] Fix for the moment $u\in\SR$. Given $n\in \SP(1,y)$ with $u<f(n)\le u+\epsilon$, we write $n=ab$, where $a$ is square-free, $b$ is square-full and $(a,b)=1$. We further decompose $a=a_1a_2a_3$, where $a_1\in \P(1,q_0)$, $a_2\in\P(q_0,y)$ and all primes factors of $a_3$ lie in $\mathcal{Q}:=\SP\setminus\P$. So
\al{
\sum_{\substack{P^+(n)\le y \\ u<f(n)\le u+\epsilon }} \frac1n
	&= \sum_{ \substack{ P^+(b)\le y \\ b\ \text{square-full}} } \frac{1}{b} 
		\sum_{\substack{a_3\in \mathcal{Q}(1,y) \\ (a_3,b)=1 }} \frac{\mu^2(a_3)}{a_3}
	 	\sum_{\substack{ a_2\in \P(q_0,y) \\ (a_2,b)=1} } \frac{\mu^2(a_2)}{a_2} \nn
	&\qquad \times \sum_{\substack{ a_1\in \P(1,q_0),\ (a_1,b)=1 \\ u - f(a_2a_3b) < f(a_1) \le u - f(a_2b) + \epsilon }} 
		\frac{\mu^2(a_1)}{a_1} \nn
	&\le  \sum_{ \substack{ P^+(b)\le y \\ b\ \text{square-full}} } \frac{1}{b} 
		\sum_{\substack{a_3\in \mathcal{Q}(1,y)  }} \frac{1}{a_3}
	 	\sum_{a_2\in \P(q_0,y) } \frac{1}{a_2} 
		\sup_{v\in\SR}\left\{\sum_{\substack{ a_1\in \P(1,q_0) \\ v < f(a_1)\le v+\epsilon}}
			\frac{\mu^2(a_1)}{a_1} \right\} \nn
	&\ll \frac{\log y}{\log q_0} \cdot   \sup_{v\in\SR}\left\{\sum_{\substack{ a_1\in \P(1,q_0) \\v<f(a_1)\le v+\epsilon}}
		\frac{\mu^2(a_1)}{a_1}\right\}.  \label{e100}
}
Next, fix $v\in\SR$ and let 
\[
\N = \{n\in\P(1,q_0): \mu^2(n)=1,\ v<f(n)\le v+\epsilon\}.
\]
As in the proof of Theorem \ref{ekc}, we split $\N$ according to the size of $\min\{p|n:p>q_J\}$. So we write $\N=\cup_{j=0}^J\N_j$, where $\N_0= \N\cap \P(1,q_J)$ and
\[
\N_j = \{n\in \N  : n=apb,\ P^+(a)\le q_J<p<P^-(b),\ q_j<p\le q_{j-1} \} .
\]
First, we bound $\sum_{n\in\N_0}1/n$. If $n>1$, then we write $n=mP^+(n)=mp'$. So we find that 
\al{
\sum_{n\in\N_0}\frac1n
	&\le 1 + \sum_{P^+(m)\le q_J}\frac1m
		\sum_{\substack{ p'\in\P\cap(P^+(m),q_J] \\ v-f(m)<f(p') \le v-f(m)+\epsilon}} \frac1{p'} \nn
	&\ll  1 + \sum_{P^+(m)\le q_J} \frac1m \left( \frac{\epsilon}{2^J\delta} + \frac1{\log^2(1+P^+(m))}\right)
		\ll1+ \frac{\epsilon\log q_J}{2^J\delta} , \label{e5}
}  
by applying the second part of\ \eqref{c2} with $v-f(m)$, $2^J\delta$, $P^+(m)$ and $q_J$ in place of $u$, $\delta$, $z$ and $w$, respectively. 

Next, we bound $\sum_{n\in\N_j}1/n$ for $j\in\{1,\dots,J\}$. In this part of the argument we may assume that $J\ge1$; otherwise, there is no such $j$. Then we have that 
\al{
\sum_{n\in\N_j}\frac1n 
	&\le\sum_{b\in\P(q_j,q_0)} \frac1b 
		\sum_{a_1\in\P (q_J^\lambda, q_J) } \frac1{a_1}
		\sum_{a_2\in\P(1,q_J^\lambda)} \frac{\mu^2(a_2)}{a_2} 
		\sum_{ \substack{p\in\P\cap(q_j,q_{j-1}] \\ v<f(p)+f(a_1b)+f(a_2) \le v + \epsilon} }
		\frac1p \nn
	&\le	\sum_{b\in\P(q_j,q_0)} \frac1b \sum_{a_1\in\P(q_J^\lambda,q_J)} \frac1{a_1} 
			\sup_{t\in\SR} \left\{ \sum_{a_2\in\P(1,q_J^\lambda)} \frac{\mu^2(a_2)}{a_2}
			\sum_{\substack{p\in\P \cap (q_j,q_{j-1}] \\ t < f(p) +f(a_2) \le t + \epsilon}}	
			\frac1p\right\}.      \label{e7}
}
Fix some $t\in\SR$ and consider $a_2\in\P(1,q_J^\lambda)$ and $p\in\P\cap(q_j,q_{j-1}]$ with $t<f(a_2)+f(p)\le t+\epsilon$, as above. Since $|f(p)|\le2^j\delta$ for $p\in\P\cap(q_j,+\infty)$, by\ \eqref{c1}, we must have that $|f(a_2) - t| \le 2^{j+1}\delta$. So 
\al{
\sum_{a_2\in \P (1,q_J^\lambda)} \frac{\mu^2(a_2)}{a_2} 
		\sum_{\substack{p\in\P\cap(q_j,q_{j-1}] \\ t<f(p)+f(a_2)\le t+\epsilon }} \frac1p 
	&=\sum_{\substack{a_2\in\P(1,q_J^\lambda) \\ |f(a_2) - t| \le 2^{j+1}\delta }} \frac{\mu^2(a_2)}{a_2}
		\sum_{\substack{p\in	\P \cap(q_j,q_{j-1}] \\ t - f(a_2) < f(p) \le t - f(a_2) + \epsilon } }  \frac1p \nn
	&\ll \frac{\epsilon} {2^j\delta} 
		\sum_{ \substack{ a_2\in\P (1,q_J^\lambda) \\ |f(a_2) - t| \le 2^{j+1}\delta }} \frac{\mu^2(a_2)}{a_2},   \label{e8}
}
by the first part of\ \eqref{c2} applied with $t-f(a_2)$, $2^{j-1}\delta$, $q_j$ and $q_{j-1}$ in place of $u$, $\delta$, $z$ and $w$, respectively, since $z\ge q_j\ge q_J = P_f(2^J\delta) \ge(2^J\delta/\epsilon)^\rho \ge (2^j\delta/\epsilon)^\rho$. Finally, if $a_2>1$, then we write $a_2=mP^+(a_2)=mp'$. Consequently
\[
\sum_{\substack{ a_2\in \P(1,q_J^\lambda) \\ |f(a_2) - w| \le 2^{j+1}\delta }} \frac{\mu^2(a_2)}{a_2}
	\le 1 + \sum_{m\in\P(1,q_J^\lambda)} \frac1m 	
	\sum_{\substack{p'\in\P\cap(P^+(m),q_J^\lambda] \\ | f(p') - (t-f(m))| \le 2^{j+1}\delta }} \frac1{p'} .
\]
For every fixed $m\in\SN$ we have that 
\[
\sum_{ \substack{ p'\in\P\cap(P^+(m),q_J^\lambda] \\ | f(p') - (t-f(m))| \le 2^{j+1} \delta }} \frac1{p'}
\ll \frac{2^j\delta}{2^J\delta} + \frac1{\log^2(1+P^+(m))},
\]
by\ \eqref{c2} with $2^j\delta$, $2^J\delta$, $P^+(m)$ and $q_J^\lambda$ in place of $\epsilon$, $\delta$, $z$ and $w$, respectively, and with $u\in\{t-f(m)+h\cdot 2^{j}\delta: h\in\{-2,-1,0,1\} \}$. So we find that 
\als{
\sum_{\substack{a_2\in\P(1,q_J^\lambda) \\ |f(a_2) - w| \le  2^{j+1}\delta }}\frac{\mu^2(a_2)}{a_2}
	&\ll 1 + \sum_{P^+(m)\le q_J^\lambda} \frac1m \left( \frac{2^j\delta}{2^J\delta} + \frac1{\log^2(1+P^+(m))}\right) \\
	&\ll 1 + \frac{1+\lambda\log q_J }{2^{J-j}}   \ll 1  + \frac{\lambda\log q_J}{2^{J-j}} .
}
Combining the above estimate with\ \eqref{e7} and\ \eqref{e8} implies that 
\als{
\sum_{n\in\N_j} \frac1n 
	&\ll \frac{\epsilon}{2^j\delta} \left(1+\frac{\lambda\log q_J }{2^{J-j}} \right) 
		\sum_{b\in\P(q_j,q_0)} \frac1b\sum_{a_1\in\P(q_J^\lambda,q_J)}\frac1{a_1} \\
	&\ll \frac{\epsilon}{2^j\delta} \left(1+\frac{\lambda\log q_J } {2^{J-j}} \right)  
		\frac{\log q_0}{\log q_j} \cdot \frac{1}{\lambda} 
	\le \frac{\epsilon }{2^j\delta} \left( \frac{1}{\lambda} + \frac{\log q_J}{2^{J-j}} \right) 
		\frac{\log q_0}{\log q_j},
}
which, together with relations \eqref{e100} and \eqref{e5}, completes the proof of the lemma.
\end{proof}


We are now in position to complete the proof of Theorem \ref{main thm}.

\begin{proof}[Proof of the upper bound in Theorem\ \ref{main thm}] As we have already seen, the function $t\to g(t)(\log t)^c$ is decreasing. In particular, $g$ is strictly decreasing. For every $\delta\in(0,1]$, we define
\[
K^*(\delta) = \min\{n\in\SN: n\ge 3,\ g(n)\le \delta\}.
\]
Then we have that $K^*(\delta)-1\le K(\delta) \le 2K^*(\delta)$, with the second inequality being a consequence of Bertrand's postulate. 

We claim that 
\eq{e10}{
\log(K^*(\delta)-1)
	\ge \frac{1}{2} \left( \frac{\eta}{\delta} \right)^{1/c} \log(K^*(\eta)-1) \quad(0<\eta \le\delta\le1).
}
Indeed, if $K^*(\delta)=K^*(\eta)$, then this inequality holds trivially. Next, assume that $K^*(\eta)\ge K^*(\delta)+1\ge4$. Then the definition of $K^*(\eta)$ implies that $g(K^*(\eta)-1)>\eta$. Since, in addition, the function $t\to g(t)(\log t)^c$ is decreasing and $(x-1)\le x^2$ for all $x\ge3$, we find that
\als{
1 \ge \frac{g(K^*(\eta)-1) (\log  (K^*(\eta)-1) )^c}{g(K^*(\delta)) (\log K^*(\delta))^c }
	&\ge \frac{\eta (\log (K^*(\eta)-1))^c}{\delta (\log K^*(\delta))^c} \\
	&\ge\frac{\eta (\log(K^*(\eta)-1))^c} {\delta (2\log(K^*(\delta)-1))^c } .
}
In any case, \eqref{e10} holds.

Using relation \eqref{e10}, we shall show that we may apply Lemma \ref{main-upper} with $P_f=K^*-1$, $\P$, $\lambda=1/A$ and $\rho=2$. Condition\ \eqref{c0} holds by assumption and condition \eqref{c1} follows immediately by the definition of $K^*$ and the fact that $|f(p)|\le g(K^*(\delta))\le\delta$ for $p\ge K^*(\delta)$. Lastly, we show \eqref{c2} with $\lambda=1/A$ and $\rho=2$. This will be done in several steps. Fix $u\in\SR$, $0<\eta\le \delta \le 1$ and $2\le z\le w\le \min\left\{ P_f(\delta),P_f(\eta)^{1/A}\right\}$. 

First, we show\ \eqref{c2} when $z\ge w^{1/4}$. By assumption, there is an absolute constant $C>0$ such that 
\[
|f(p_1)-f(p_2)|
	\ge \frac2C\min\left\{\frac{g(p_2)(p_2-p_1)}{p_2\log p_2},g\left(p_2^A\right)\right\} 
		\quad(p_1<p_2,\,p_1,p_2\in\SP).
\]
We claim that if $v\in\SR$ and $\eta' := \min\{\eta,\delta/\log w\}/C$, then
\eq{e11}{
\sum_{\substack{p\in\P\cap(z,w] \\ v<f(p)\le v + \eta' }} \frac1p
	\ll\begin{cases}
	 	\eta'/\delta
			 &\text{if}\ z\ge P_f(2\delta) \ge (2\delta/\eta)^2  , \\
		\eta'/\delta + 1/(\log w)^3  		
				&\text{otherwise}.
	\end{cases}
}
If this relation does hold, then breaking the interval $(u,u+\eta]$ into at most $1+\eta/\eta' \le1+C\log w$ intervals of the form $(v,v+\eta']$, we deduce that\ \eqref{c2} holds too when $z\ge w^{1/4}$. So it remains to show \eqref{e11} to complete the proof of \eqref{c2} in this special case.

Without loss of generality, we may assume that $w\ge3$; otherwise there are no primes in $(z,w]\subset(2,3)$ and \eqref{e11} is trivially true. In particular, we may assume that $K^*(\eta)\ge K^*(\delta)\ge 4$. Therefore, for every $p\in(z,w]$, we have that $g(p)\ge g(K^*(\eta)-1)>\eta$ and $g(p^A)\ge g(K^*(\delta)-1)>\delta$. Now, consider two primes $p_1<p_2$ that both belong to the set $\{p\in\P\cap(z,w]: v<f(p)\le v+\eta'\}$. Then
\eq{e12}
{
\frac1C\min\left\{\eta,\frac{\delta}{\log w}\right\} 
	=\eta'
	&>|f(p_1)-f(p_2)| \\
	&\ge\frac2C\min\left\{ \frac{g(p_2)(p_2-p_1)}{p_2\log p_2},g\left(p_2^A\right)\right\}\\
	&\ge\frac2C\min\left\{ \frac{\delta (p_2-p_1)}{p_2\log p_2}, \eta \right\} 
}
and, consequently,
\[
0<p_2-p_1 \le  \frac{C\eta'}{2\delta} \cdot   p_2\log p_2
	= \min\left\{ \frac{\eta}{\delta}, \frac1{\log w} \right\}  \cdot  \frac{p_2\log p_2}{2} \le \frac{p_2}2.
\]
Set 
\[
P=\max\{p\in\P\cap(z,w]:v<f(p)\le v+\eta'\}
\]
and
\[
y= \frac{C\eta'}{2\delta}  \cdot  P\log P  \le \frac{P}{2},
\]
so that $\{p\in\P\cap(z,w]: v<f(p)\le v+\eta'\}\subset[P-y,P]$. The second part of relation\ \eqref{e11} then follows by the Prime Number Theorem\ \cite[Theorem 1, p. 167]{ten_book}. For the first part of\ \eqref{e11}, note that if $z\ge P_f(2\delta) \ge (2\delta/\eta)^2$, then 
\[
\frac{y}{\sqrt{P}} 
	= \frac{\sqrt{P}\log P}{2}  \min\left\{ \frac{\eta}{\delta}, \frac1{\log w} \right\}
	\ge \frac{(2\delta/\eta)\log z }{2} \min\left\{ \frac{\eta}{\delta} , \frac1{\log w} \right\} 
	 \ge \frac{1}{4} ,
\]
where we used our assumption that $z\ge \max\{w^{1/4},2\}$. So the first part of\ \eqref{e11} follows by the Brun-Titchmarsch inequality\ \cite[Theorem 9, p. 73]{ten_book}, thus completing the proof of\ \eqref{e11} and hence of \eqref{c2} in the case when $z\ge w^{1/4}$.

Finally, we show \eqref{c2} when $z<w^{1/4}$. First, note that 
\[
\log P_f(2^j\delta) \ge 2^{-1-j/c}\log P_f(\delta)\ge2^{-1-j/c} \log w \ge4^{-j}\log w,
\]
by \eqref{e10}, for every $j\ge1$. Since $w\le P_f(\delta)$ too, by assumption, we deduce that
\eq{w j}{
P_f(2^j\delta) \ge w_j:= w^{4^{-j}} \quad (j \ge 0).
}
Applying this inequality with $j=1$ implies that $z<P_f(2\delta)$, that is to say we are in the second case of \eqref{c2}. Let 
\[
j_0 =\max\{j\ge 0: \mbox{$w_j\ge z$ and $2^j\le 1/\delta$}\} ,
\]
\[
S_j = \sum_{\substack{p\in\P\cap(w_{j+1},w_j] \\ u<f(p)\le u+\eta }}\frac1p ,
\]
for $j\in\{0,1,\dots,j_0-1\}$, and
\[
S_{j_0} = \sum_{\substack{p\in\P\cap(z,w_{j_0}] \\ u<f(p)\le u+\eta }}\frac1p .
\]
Then the part of\ \eqref{c2} that we have already proven and \eqref{w j} imply that
\eq{Sj estimate}{
S_j \ll  \frac{\eta}{2^j\delta}+\frac{16^j}{(\log w)^2} ,
}
for $j\in\{0,1,\dots,j_0-1\}$. We claim that the same estimate holds for $S_{j_0}$ too. If $2^{j_0+1}\delta\le1$, then $w_{j_0}^{1/4}=w_{j_0+1}<z$ and thus we may apply again the part of\ \eqref{c2} that we have already proven. Finally, if $2^{j_0+1}\delta>1$, then we have that $w_{j_0}\le P_f(2^{j_0}\delta)\le P_f(1/2)\ll 1$, since $g(t)\ll1/(\log t)^c$ by our assumptions on $f$. Consequently, covering the interval $(z,w_{j_0}]$ by $O(1)$ intervals of the form $(t,t^4]$ and applying the already proven part of\ \eqref{c2} shows that \eqref{Sj estimate} holds in this case too for $j=j_0$. Summing \eqref{Sj estimate} over $j\in\{0,1,\dots,j_0\}$ implies that
\[
\sum_{\substack{p\in\P\cap(z,w]\\   u<f(p)\le u+\eta}}\frac1p
= \sum_{j=0}^{j_0} S_j \ll \frac{\eta}{\delta} + \frac{1}{(\log z)^2} ,
\]
which completes the proof of \eqref{c2}. In conclusion, we may apply Theorem\ \ref{main-upper} with $P_f=K^*-1$, $\rho=2$ and $\lambda=1/A$.

We are finally ready to show the upper bound in Theorem \ref{main thm}. Let $\epsilon\in(0,1/2]$. We may assume that $K^*(\epsilon)$ is large enough; otherwise, the theorem follows by the trivial upper bound $Q_{\mathcal{F}_y}(\epsilon)\le1$. In particular, we may assume that the parameter $\delta:=g\left(\fl{ K^*(\epsilon)^{1/A} } -1 \right)$ lies in $[\epsilon,1/2]$. Since $g$ is strictly decreasing, the definition of $K^*$ implies that 
\eq{e400}{
K^*(\delta) = \fl{ K^*(\epsilon)^{1/A} } -1 .
} 
In particular, $P_f(\delta)\le P_f(\epsilon)^{1/A}$. For $j\in\SN\cup\{0\}$ with $2^j\le1/\delta$, we set $q_j=P_f(2^j\delta)=K^*(2^j\delta)-1$. Note that 
\[
q_0 = K^*(\delta) - 1  \le K^*(\epsilon) -1 \le K(\epsilon)
\]
and
\eq{e6}{
\log q_j  \ge2^{-1-(j-i)/c}\log q_i  \quad(0\le i\le j \le \log(1/\delta) /\log 2) ,
}
by\ \eqref{e10}. Set 
\[
J=\max\left(\{0\}\cup\{j\in\SN:2^j\le1/\delta\ \text{and}\ q_j\ge(2^j\delta/\epsilon)^2\}\right).
\]
Then Theorem\ \ref{main-upper} and relation\ \eqref{e6} imply that, for $y\ge K(\epsilon) \ge q_0$, we have that
\al{
Q_{\mathcal{F}_y}(\epsilon)
	&\ll_A\frac1{\log q_0}+\frac{\epsilon}{\delta} \sum_{j=1}^J\frac1{2^j\log q_j}
		+\frac\epsilon\delta\sum_{j=0}^J\frac{\log q_J}{2^J\log q_j}    \nn
	&\ll \frac1{\log q_0} + \frac{\epsilon}{\delta} \sum_{j=1}^J\frac1{2^{j-j/c}\log q_0}
		+ \frac{\epsilon}{\delta}\sum_{j=0}^J\frac{\log q_J}{2^{J-j/c}\log q_0}    \nn
	&\ll \frac{\min\left\{1/(c-1),1+J\epsilon/\delta\right\}}{\log q_0}
		+\frac\epsilon\delta\frac{\log q_J}{2^{J(1-1/c)}\log q_0}     \nn
	&\ll \frac{\min\left\{1/(c-1),1 + J\epsilon/\delta\right\}}{\log q_0} 
		+\frac\epsilon\delta\frac{\log q_J}{\max\{1,(J+1)(c-1)\}\log q_0}.    \label{e3}
}
Finally, note that if $2^{J+1}\le1/\delta$, then the maximality of $J$ and\ \eqref{e6} imply that 
\[
\log q_J\le4\log q_{J+1} \le 8 \log(2^{J+1}\delta / \epsilon) \ll J+1+\log(\delta/\epsilon).
\] 
On the other hand, if $2^{J+1}>1/\delta$, then $q_J\le K^*(1/2)\ll1$, since $g(t)\ll(\log t)^{-c}$. In any case, we find that $\log q_J \ll J+1+\log(\delta/\epsilon)$. So the inequalities
\[
\frac{J\epsilon}{\delta} 
	\ll \frac{\epsilon\log(1/\delta)}{\delta}
	\le \log(1/\epsilon) 
\quad\text{and}\quad
	\frac{\epsilon}{\delta} \log(\delta/\epsilon) \ll 1
\]
and relation \eqref{e3} imply that
\[
Q_{\mathcal{F}_y}(\epsilon)  \ll   \frac{ \min\{1/(c-1),\log(1/\epsilon)\}}{\log q_0} .
\]
Finally, we have that $\log q_0\asymp_A \log K^*(\epsilon) \asymp \log K(\epsilon)$, by \eqref{e400} and the fact that $K^*-1\le K\le 2K^*$. So the upper bound in Theorem \ref{main thm} follows.
\end{proof}


\section*{Acknowledgents}

I would like to thank G\'erald Tenenbaum for pointing out relation\ \eqref{c<1} to me. I am also grateful to R\'egis de la Bret\`eche and Maksym Radziwill for some helpful comments. In addition, I would like to thank the referees who handled the paper, as their comments exposed some inaccuracies and improved the exposition of the main ideas. This paper was largely written while visiting Universit\'e Paris-Diderot, which I would like to thank for its hospitality. During that time I was a postdoctoral fellowship at the Centre de recherches math\'ematiques at Montr\'eal, which I would like to thank for the financial support.


\bibliographystyle{amsplain}

\end{document}